\documentclass[a4paper, 12 pt, reqno]{amsart}
\usepackage{amsfonts, amssymb, amsmath, eucal, amsthm,verbatim}
\usepackage[all]{xy}
\usepackage[usenames]{color} 

\newtheorem{lemma}{Lemma}
\newtheorem{theorem}[lemma]{Theorem}
\newtheorem*{theorem*}{Theorem}
\newtheorem*{TheoremA}{Theorem A}
\newtheorem*{TheoremB}{Theorem B}

\newtheorem{corollary}[lemma]{Corollary}

\theoremstyle{definition}
\newtheorem{definition}[lemma]{Definition}
\newtheorem*{definition*}{Definition}

\newtheorem{example}[lemma]{Example}

\newcommand{\nOrd}{n\Ord}
\newcommand{\Ord}{\mathrm{Ord}}
\newcommand{\calC}{\mathcal{C}}
\newcommand{\calD}{\mathcal{D}}

\newcommand{\calP}{\mathcal{P}}

\newcommand{\RR}{\mathbb{R}}

\DeclareMathOperator{\hocolim}{\mathrm{hocolim}}

\newcommand{\conf}{\mathrm{Conf}}

\usepackage{mathptmx}
\DeclareSymbolFont{cmlargesymbols}{OMX}{cmex}{m}{n}
\usepackage[vmargin=3.2cm, hmargin=3.4cm]{geometry}
\DeclareMathOperator{\colim}{\mathrm{colim}}

\title{Configuration spaces and $\Theta_n$}
\author{David Ayala}
\address{Department of Mathematics\\
Harvard University\\
One Oxford Street\\
Cambridge, MA 02138 USA}
\email{davidayala.math@gmail.com}
\author{Richard Hepworth}
\address{Institute of Mathematics\\
University of Aberdeen\\
Aberdeen AB24 3UE\\
United Kingdom}
\email{r.hepworth@abdn.ac.uk}
\thanks{Supported by the Danish National Research Foundation (DNRF) through the Centre for Symmetry and Deformation.
The first author was partially supported by ERC adv.grant no.228082, and by the National Science Foundation under Award No. 0902639}
\date{}
\begin{document}

\begin{abstract}
\noindent We demonstrate that Joyal's category $\Theta_n$, which is central to numerous definitions of $(\infty,n)$-categories, naturally encodes the homotopy type of configuration spaces of marked points in $\mathbb{R}^n$.
This article is largely self-contained and uses only elementary techniques in combinatorics and homotopy theory.
\end{abstract}

\maketitle

\section{Introduction}
\label{Introduction:Section}

Among the many approaches to a theory of $(\infty,1)$-categories,
two of the more developed are Joyal's theory of quasi-categories \cite{JoyalQuasi}
and Rezk's theory of \emph{complete Segal spaces} \cite{RezkCSS}.
In the former approach, an $(\infty,1)$-category is a contravariant functor
from the simplicial category $\Delta$ into sets,
while in the latter it is a contravariant functor from $\Delta$ into spaces;
in each case this functor is required to satisfy certain conditions.

Numerous deep and natural questions involving higher category theory, for instance the cobordism hypothesis of Baez and Dolan \cite{BaezDolan, LurieTFT}, require a developed theory of $(\infty,n)$-categories for $n\geqslant 0$.
It was in order to initiate such a theory of $(\infty,n)$-categories that Joyal introduced the categories $\Theta_n$ with $\Theta_1 = \Delta$.  Indeed, he defined an $(\infty,n)$-category to be a contravariant
functor from $\Theta_n$ into sets, directly generalising the notion of a quasi-category~\cite{JoyalDiscs}.
Thereafter, Rezk formulated a different notion of $(\infty,n)$-category as a contravariant functor from $\Theta_n$ into spaces, directly generalising the theory of complete Segal spaces~\cite{RezkCartesian}.

At present the categories $\Theta_n$ appear in several places in the higher category theory literature, and likewise admit various definitions.
We will use Berger's definition of $\Theta_n$ as the \emph{$n$-fold wreath product} of $\Delta$ with itself \cite{BergerTwo}.
This notion of wreath product is important in Lurie's theory of $\infty$-operads~\cite{LurieHA}.

The purpose of this paper is to demonstrate that the category $\Theta_n$
very naturally encodes properties of an important class of topological spaces, namely the spaces of configurations of $r$ marked points in Euclidean space $\RR^n$. Such configuration spaces arise in various situations throughout algebraic and geometric topology.
For one, when $n=2$ this configuration space is precisely the classifying space of the pure braid group on $r$ strands.
Secondly, keeping $r$ fixed and taking the limit
as $n\to\infty$ and forgetting the markings one obtains the classifying space of the symmetric group on
$r$ letters.
Lastly, the configuration space of $r$ marked points in
$\mathbb{R}^n$
is homotopy equivalent to the space of $r$-ary operations of the $E_n$-operad.

Let us introduce some terminology before stating our main result.
Fix a finite set $A$.
Recall that there is a natural assembly functor $\gamma_n\colon\Theta_n\to\Gamma$ taking values in Segal's category of finite sets.

\begin{definition*}
Let $\Theta_n(A)$ denote the following category.
The objects are pairs $(S,f)$ where $S$ is an object of $\Theta_n$ and $\sigma\colon\gamma_n(S)\to A$ is an isomorphism.
A morphism $(S,\sigma)\to (T,\tau)$
is a morphism $\lambda\colon S\to T$ in $\Theta_n$ for which $\tau\circ\gamma_n(\lambda)=\sigma$.
\end{definition*}

\begin{definition*}
Let $\conf_A(\RR^n)$ denote the space of all injective functions $A\to\RR^n$.
When $A=\{1,\ldots,r\}$ this is simply the space of configurations of $r$ marked points in $\RR^n$.
\end{definition*}

\begin{theorem*}
There is a homotopy equivalence
\[B(\Theta_n(A))\simeq \conf_A(\RR^n)\]
between the classifying space of $\Theta_n(A)$ and the configuration space $\conf_A(\RR^n)$.
\end{theorem*}

It is our hope that this paper will be accessible to category theorists new to the configuration space $\conf_A(\RR^n)$, and also to topologists new to the category $\Theta_n$.
In particular we do not rely on any of the literature on
configuration spaces, and we do not assume any prior knowledge of the category $\Theta_n$.

Both $B(\Theta_n(A))$ and $\conf_A(\RR^n)$ admit evident free actions of the permutation group $\Sigma_A$, and the equivalence $B(\Theta_n(A))\simeq \conf_A(\RR^n)$ is in fact a $\Sigma_A$-equivariant homotopy equivalence.  It is our expectation that this equivalence extends to a more general statement in which the set $A$ is allowed to vary not just by bijections, but by surjections.  More concretely, we expect to show in future work that (a simplicial localisation of) $\Theta_n$ is equivalent to (the exit-path category of a non-compact version of) the Ran space of $\RR^n$.
A consequence of such a result would be an explicit comparison between $\Theta_n$-spaces (that is, contravariant functors from $\Theta_n$ to spaces) and $E_n$-algebras  which would make use of `factorisation algebras' in the sense of Lurie~(\cite{LurieHA}).
Such a comparison is to be expected.  For instance, Berger (\cite{BergerTwo}) has shown that group-like reduced $\Theta_n$-spaces are a `model' for $n$-fold loop spaces.  
\\

Let us sketch the proof of the theorem.
We will make use of an elementary combinatorial object which we call the \emph{poset of $n$-orderings of $A$} and  write as $\nOrd(A)$.
The elements of $\nOrd(A)$ are certain trees of height $n$ with leaves
labelled by $A$, and the partial order
is determined by a simple criterion that we call the \emph{branching condition}.
When $n=1$ the poset $\nOrd(A)$ is simply the set of linear orderings of $A$,
and the partial ordering is the trivial one.
The poset of $n$ orderings appears elsewhere in other other guises:
in \cite{Batanin}
it is the poset of total complementary $n$-orders on $A$,
while in \cite{BFSV} it is the poset of $|A|$-ary operations
in the Milgram preoperad.
The relevance of $\nOrd(A)$ is that it mediates between $\Theta_n(A)$ and $\conf_A(\RR^n)$:

\begin{TheoremA}
There is a homotopy equivalence $B(\nOrd(A))\simeq \conf_A(\RR^n)$.
\end{TheoremA}

\begin{TheoremB}
There is a full embedding $\nOrd(A)\hookrightarrow \Theta_n(A)$ which induces a homotopy equivalence on geometric realisations.
\end{TheoremB}

These two theorems together imply the main result above.

Theorem~A is related to the \emph{Fox-Neuwirth cell decomposition} of $\conf_A(\RR^n)$.
Fox and Neuwirth exhibited in~\cite{FoxNeuwirth} a decomposition of $\conf_A(\RR^2)$ into finitely many open cells.
This is not a CW-decomposition, but still the topological boundary of each cell meets only cells of lower dimension.
Consequently the cells themselves form the elements of a poset.
The generalisation to $n>2$ is discussed 
in \cite[Section~5.4]{GetzlerJones},
\cite[section~6]{Batanin}
and \cite{GiustiSinha}.
From this point of view, Theorem~A identifies the poset of Fox-Neuwirth cells with $\nOrd(A)$, and shows that its realisation is homotopy-equivalent to $\conf_A(\RR^n)$ itself.

Theorem~A is contained in a theorem of
Balteanu, Fiedorowicz, Schw\"anzl and Vogt \cite[Theorem~3.14]{BFSV}.
However, we present our own proof of the result, both in
order to give a self-contained account of our main theorem, and because
our proof is significantly simpler.
(This is not surprising: Theorem~A appears in \cite{BFSV} as just one part
of a more elaborate result.)

Theorem~B amounts to a careful study of the morphisms of $\Theta_n$.
For it is well-known that the objects of $\Theta_n$ admit a simple description as
the \emph{planar level trees of height $n$}.
Thus it is simple to construct the claimed embedding
$\nOrd(A)\hookrightarrow\Theta_n(A)$ on the level of objects.
However, given objects $S$ and $T$ of $\Theta_n$, described as planar level trees,
it is difficult to describe the collection of all morphisms $S\to T$ in a
similarly combinatorial way.
(This can be seen as one of the causes for the profusion of definitions
of $\Theta_n$ itself.)
Nevertheless,
we are able to prove a theorem that gives a simple combinatorial description of the set of all \emph{active} morphisms $S\to T$ when $T$ is \emph{healthy}.
(Here \emph{active} and \emph{healthy} are appropriate restricted classes
of morphisms and objects.)
This is sufficient to prove Theorem~B.

The paper is arranged as follows.
In section~\ref{nOrd:Section} we introduce the poset $\nOrd(A)$,
then in section~\ref{nOrdHomotopy:Section} we compute its homotopy type
and prove Theorem~A.
Section~\ref{Definitions:Section} recalls the definition of $\Theta_n$ in detail.
Then section~\ref{Morphisms:Section} proves our characterisation of the
active morphisms $S\to T$ when $T$ is healthy.
This result is then used in section~\ref{ThetanA:Section} to prove Theorem~B.

\subsection*{Acknowledgments}

This work began in the Topology Reading Seminar in
the mathematics department of Copenhagen University.
Together with the other members of the topology group, we spent several weeks
in 2010 studying Clemens Berger's papers~\cite{BergerOne}
and~\cite{BergerTwo}, which inspired the results presented here.
We would like to thank Berger for his work, and we would like to thank
the other members of Copenhagen's topology group for their friendship
and support.

\section{The poset of $n$-orderings on $A$}\label{nOrd:Section}

As in the introduction, we fix a finite set $A$ and an integer $n\geqslant 1$.
This section will first define the notion of an \emph{$n$-ordering} on $A$, and then make the set of $n$-orderings on $A$ into a poset.
(For us, a \emph{poset} is a category in which for each pair of objects $c,d$ there is at most one morphism $c\to d$, and in which the only isomorphisms are the identity morphisms.)
As explained in the introduction, this poset appears elsewhere in the literature, in particular in \cite{Batanin} and \cite{BFSV}.
Here we will introduce the poset from scratch.

\begin{definition}\label{PlanarLevelTrees:Definition}
A \emph{level tree} is a tree equipped with a preferred vertex called the \emph{root}.  The root gives a preferred direction to all edges of the tree, and there is a unique directed path from any vertex to the root.  
The \emph{incoming} edges at a vertex are those edges directed toward the vertex.
A \emph{planar level tree} is a level tree equipped with a linear ordering on the {incoming} edges at each vertex.  
The \emph{level} of a vertex is the length of the directed path from that vertex to the root.
A tree has \emph{height $n$} if the maximum of the levels of its vertices is at most $n$.  Notice then that a tree of height $n$ is a tree of height $n+k$ for any $k\geqslant 0$.
A vertex $v$ is called a \textit{leaf} if it has no incoming edges.
We depict planar level trees as a diagrams as follows,
with the root at the bottom and with the linear orderings read from left to right.
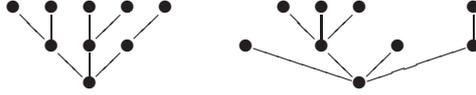
\[
\xy
(0,0)*{\bullet}="A";
(-5,5)*{\bullet}="B1";
(-0,5)*{\bullet}="B2";
(5,5)*{\bullet}="B4";
(-10,10)*{\bullet}="C1";
(-5,10)*{\bullet}="C2";
(0,10)*{\bullet}="C3";
(5,10)*{\bullet}="C4";
(10,10)*{\bullet}="C6";
"A";"B1" **\dir{-};
"A";"B2" **\dir{-};
"A";"B4" **\dir{-};
"B1";"C1" **\dir{-};
"B1";"C2" **\dir{-};
"B1";"C2" **\dir{-};
"B2";"C3" **\dir{-};
"B2";"C4" **\dir{-};
"B4";"C6" **\dir{-};
\endxy
\quad\quad
\xy
(0,0)*{\bullet}="A";
(-15,5)*{\bullet}="B1";
(-5,5)*{\bullet}="B2";
(5,5)*{\bullet}="B3";
(15,5)*{\bullet}="B4";
(-10,10)*{\bullet}="C3";
(-5,10)*{\bullet}="C4";
(0,10)*{\bullet}="C5";
(15,10)*{\bullet}="C6";
"A";"B1" **\dir{-};
"A";"B2" **\dir{-};
"A";"B3" **\dir{-};
"A";"B4" **\dir{-};
"B2";"C3" **\dir{-};
"B2";"C4" **\dir{-};
"B2";"C5" **\dir{-};
"B4";"C6" **\dir{-};
\endxy
\] 
\end{definition}

\begin{definition}
A planar level tree of height $n$ is \emph{healthy} if it has no leaves at levels $1,\ldots,n-1$.  In the illustration above the first tree is healthy of height $2$, whereas the second tree is not healthy.
\end{definition}

Note that the question of whether a planar level tree of height $n$ is healthy is dependent on $n$.
The tree of height $0$, which consists of the root and nothing more, is healthy regardless of the value of $n$.
This might seem anomalous, but it will allow the empty set to admit a (unique) $n$-ordering for each $n$, as we see now.

\begin{definition}\label{nOrderings:Definition}
An \emph{$n$-ordering} on $A$ is a pair $(S,\sigma)$ consisting of a healthy planar level tree $S$ of height $n$, together with a bijection $\sigma$ between $A$ and the level-$n$ leaves of $S$.
We will usually denote an $n$-ordering $(S,\sigma)$ by $S$ alone, leaving the bijection $\sigma$ implicit.
\end{definition}

\begin{example}\label{nOrderings:ExampleOne}
A $1$-ordering on $A$ is precisely a linear ordering on $A$.
\end{example}

\begin{example}\label{nOrderings:ExampleTwo}
There are exactly {four} $2$-orderings on the set $A=\{a,b\}$, and they are depicted below.
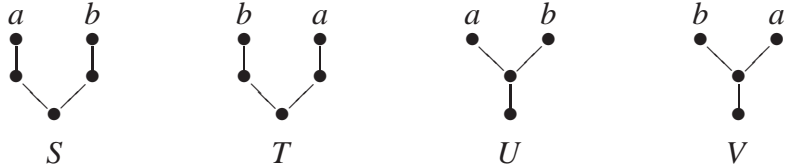
\[\xy
(-60,0)*{\bullet}="T1";
(-60,-5)*{S};
(-65,5)*{\bullet}="T2";
(-55,5)*{\bullet}="T3";
(-65,10)*{\bullet}="T4";
(-65,13)*{a};
(-55,10)*{\bullet}="T5";
(-55,13.5)*{b};
"T1";"T2" **\dir{-};
"T1";"T3" **\dir{-};
"T2";"T4" **\dir{-};
"T3";"T5" **\dir{-};
(-30,0)*{\bullet}="T1";
(-30,-5)*{T};
(-35,5)*{\bullet}="T2";
(-25,5)*{\bullet}="T3";
(-35,10)*{\bullet}="T4";
(-35,13.5)*{b};
(-25,10)*{\bullet}="T5";
(-25,13)*{a};
"T1";"T2" **\dir{-};
"T1";"T3" **\dir{-};
"T2";"T4" **\dir{-};
"T3";"T5" **\dir{-};
(0,0)*{\bullet}="T1";
(0,-5)*{U};
(0,5)*{\bullet}="T2";
(-5,10)*{\bullet}="T3";
(-5,13)*{a};
(5,10)*{\bullet}="T4";
(5,13.5)*{b};
"T1";"T2" **\dir{-};
"T2";"T3" **\dir{-};
"T2";"T4" **\dir{-};
(30,0)*{\bullet}="T1";
(30,-5)*{V};
(30,5)*{\bullet}="T2";
(25,10)*{\bullet}="T3";
(25,13.5)*{b};
(35,10)*{\bullet}="T4";
(35,13)*{a};
"T1";"T2" **\dir{-};
"T2";"T3" **\dir{-};
"T2";"T4" **\dir{-};
\endxy\]
\end{example}

Next, we wish to define a notion of \emph{morphism} between different $n$-orderings.
In order to do so we introduce a little more notation.

\begin{definition}\label{OrderBranchingLevels:DefinitionOne}
Let $S$ be an $n$-ordering on $A$.
The leaves of $S$ inherit a canonical linear order, and this induces a linear ordering on $A$ that we denote by $<_S$.
Given $a,b\in A$, the \emph{branching level}
\[b_S(a,b)\]
is defined to be the level of the vertex at which the directed paths from $a$ and $b$ to the root first meet. 
\end{definition}

\begin{example}\label{nOrderings:ExampleThree}
Let us return to the four elements of $2\Ord(\{a,b\})$, which were listed in Example~\ref{nOrderings:ExampleTwo}.
For the orderings we have
\[a<_S b,\qquad b<_T a,\qquad a<_U b,\qquad b<_V a\]
and for the branching levels we have
\[b_S(a,b)=0,\qquad b_T(a,b)=0,\qquad b_U(a,b)=1,\qquad b_V(a,b)=1.\]
\end{example}

\begin{definition}\label{BranchingCondition:DefinitionOne}
Let $S$ and $T$ be $n$-orderings on $A$.
The \emph{branching condition} for a morphism $S\to T$ states that the following criterion holds for all $a,b\in A$.
\begin{quotation}
We have $b_T(a,b)\leqslant b_S(a,b)$, with equality only if the ordering of $a$ and $b$ under $<_T$ agrees with the ordering of $a$ and $b$ under $<_S$.
\end{quotation}
The condition on orderings means that ($a<_T b$ and $a<_S b$) or ($b<_T a$ and $b<_S a$).
\end{definition}

\begin{definition}\label{nOrd:Definition}
The \emph{poset of $n$-orderings on $A$}, denoted $\nOrd(A)$, is the poset whose objects are the $n$-orderings on $A$, and in which there is a morphism $S\to T$ if and only if the branching condition holds.
\end{definition}

It is trivial to verify that $\nOrd(A)$ is indeed a poset.  In other words
\begin{itemize}
\item for every $n$-ordering $S$ there is a morphism $S\to S$,
\item if there are morphisms $S\to T$ and $T\to U$, then there is a morphism $S\to U$,
\item if there are morphisms $S\to T$ and $T\to S$ then $S=T$.
\end{itemize}

\begin{example}\label{nOrderings:ExampleFour}
Let us return again to $2\Ord(\{a,b\})$, as in Examples~\ref{nOrderings:ExampleTwo} and \ref{nOrderings:ExampleThree}.
There are exactly four objects, and there are also exactly four non-identity morphisms, which we depict below.
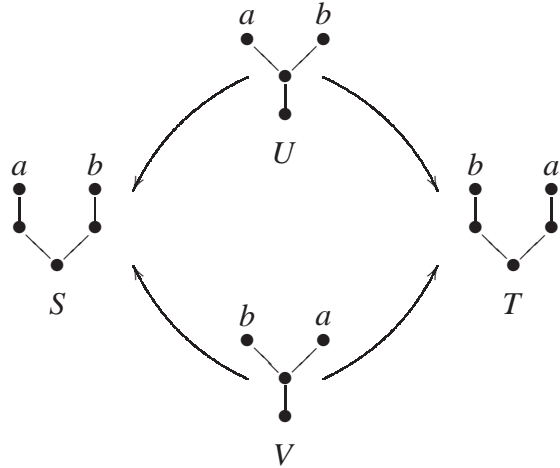
\[\xy
(0,20)*{\bullet}="T1";
(0,15)*{U};
(0,25)*{\bullet}="T2";
(-5,30)*{\bullet}="T3";
(-5,33)*{a};
(5,30)*{\bullet}="T4";
(5,33.5)*{b};
"T1";"T2" **\dir{-};
"T2";"T3" **\dir{-};
"T2";"T4" **\dir{-};
(0,-20)*{\bullet}="T1";
(0,-25)*{V};
(0,-15)*{\bullet}="T2";
(-5,-10)*{\bullet}="T3";
(-5,-6.5)*{b};
(5,-10)*{\bullet}="T4";
(5,-7)*{a};
"T1";"T2" **\dir{-};
"T2";"T3" **\dir{-};
"T2";"T4" **\dir{-};
(-30,0)*{\bullet}="T1";
(-30,-5)*{S};
(-35,5)*{\bullet}="T2";
(-25,5)*{\bullet}="T3";
(-35,10)*{\bullet}="T4";
(-35,13)*{a};
(-25,10)*{\bullet}="T5";
(-25,13.5)*{b};
"T1";"T2" **\dir{-};
"T1";"T3" **\dir{-};
"T2";"T4" **\dir{-};
"T3";"T5" **\dir{-};
(30,0)*{\bullet}="T1";
(30,-5)*{T};
(25,5)*{\bullet}="T2";
(35,5)*{\bullet}="T3";
(25,10)*{\bullet}="T4";
(25,13.5)*{b};
(35,10)*{\bullet}="T5";
(35,13)*{a};
"T1";"T2" **\dir{-};
"T1";"T3" **\dir{-};
"T2";"T4" **\dir{-};
"T3";"T5" **\dir{-};
{\ar@/_1 ex/(-5,25)*{};(-20,10)*{}};
{\ar@/^1 ex/(5,25)*{};(20,10)*{}};
{\ar@/^1 ex/(-5,-15)*{};(-20,0)*{}};
{\ar@/_1 ex/(5,-15)*{};(20,0)*{}};
\endxy\]
For example there is a morphism $U\to T$ since $b_T(a,b)<b_U(a,b)$, and there is no morphism $U\to V$ since $b_U(a,b)=b_V(a,b)$ while $a<_Ub$ and $a>_V b$.

Observe that the classifying space $B(2\Ord(\{a,b\})$ is homeomorphic to $S^1$, which is homotopy equivalent to the configuration space $\conf_2(A)$ of two labelled points in $\mathbb{R}^2$.
\end{example}

\section{The homotopy type of $\nOrd(A)$}
\label{nOrdHomotopy:Section}

Now we turn to the proof of Theorem~A, which states that there is a homotopy equivalence
$B(\nOrd(A))\simeq \conf_A(\RR^n)$.
As explained in the introduction, the key to the proof is
that $\nOrd(A)$ is the poset indexing the `cells' in the Fox-Neuwirth 
decomposition of $\conf_A(\RR^n)$ \cite{FoxNeuwirth,GetzlerJones,Batanin,GiustiSinha},
and the theorem itself is contained in \cite[Theorem~3.14]{BFSV}.
The proof we give here does not depend on any of this literature.

\begin{definition}\label{Cells:Definition}
Let $S$ be an object of $\nOrd(A)$.
Define $C(S)\subset\conf_A(\RR^n)$ to be the space of injections $\phi\colon A\hookrightarrow\mathbb{R}^n$
such that for each pair $a,b\in A$ with $a<_S b$, we have
\begin{enumerate}
\item
$\phi(a)_i=\phi(b)_i$ for $i=1,\ldots,b_S(a,b)$;
\item
\label{Second}
$\phi(a)_i\leqslant \phi(b)_i$ for $i=b_S(a,b)+1$.
\end{enumerate}
\end{definition}

Inspecting the branching condition, observe that $S\to T$ implies $C(S) \subset C(T)$.  In this way, the assignment $S\mapsto C(S)$ defines a functor
\[
C\colon \nOrd(A) \to \mathrm{Top}.
\]

\begin{lemma}\label{Sphi:Lemma}
Let $\phi\in\conf_A(\RR^n)$.
Then there is an object $S_\phi$ in $\nOrd(A)$ with the property that $\phi\in C(S)$ if and only if there is a morphism $S_\phi\to S$.
\end{lemma}
\begin{proof}
The lexicographic ordering on $\phi(A)\subset\mathbb{R}^n$ induces an  ordering on $A$ itself that we denote by $<$.
For $a,b\in A$ we define $b(a,b)$ to be the largest integer $i$ for which $\phi(a)_i=\phi(b)_i$.
There is a unique object $S_\phi$ in $\nOrd(A)$ with ordering $<$ and branching levels $b(a,b)$.
It is now trivial to check that $S_\phi$ has the required property.
\end{proof}

\begin{lemma}
\label{Colimit:Lemma}
The colimit of $C\colon\nOrd(A)\to\mathrm{Top}$ is homeomorphic to $\conf_A(\RR^n)$.
\end{lemma}
\begin{proof}
The inclusions $C(S)\hookrightarrow\conf_A(\RR^n)$ determine a natural transformation from $C$ to the constant functor with value $\conf_A(\RR^n)$.
This induces a map $f\colon\colim(C)\to\conf_A(\RR^n)$.
Define a function $g\colon \conf_A(\RR^n)\to\colim(C)$ by specifying that each $g|_{C(S)}$ is the tautological map $C(S)\to\colim(C)$.
By Lemma~\ref{Sphi:Lemma} it is well-defined.
Since $\nOrd(A)$ is finite and each $C(S)\subset\conf_A(\RR^n)$ is closed, $g$ is continuous.
Finally note that $g$ is inverse to $f$.
This completes the proof.
\end{proof}

\begin{lemma}
The natural map $\hocolim(C)\to\colim(C)$ is a homotopy equivalence.
\end{lemma}
\begin{proof}
We will prove this using the method of Proposition~13.4 of \cite{Dugger}.
Define the \emph{degree} of an object of $\nOrd(A)$ to be the number of edges of the corresponding tree.
Non-identity morphisms raise the degree, so this makes $\nOrd(A)$ into a \emph{directed Reedy category}.
Then it will suffice to show that for each object $S$ the natural map
\[
L_S(C)
\to
C(S)
\]
from the \emph{latching object} 
\[L_S(C)= \colim_{T\to S,\deg(T)<\deg(S)}(C(T))\]
is a cofibration.
By adapting the proof of Lemma~\ref{Colimit:Lemma}, one can
see that this latching object is the subspace of $C(S)$ consisting of all
$\phi\colon A\hookrightarrow\RR^n$ that satisfy the conditions of Definition~\ref{Cells:Definition}, and for which at least one of the inequalities \eqref{Second} is an equality.
Thus $L_S(C) \to C(S)$ is the inclusion into an (unbounded) convex polyhedron of its boundary, and in particular is a cofibration.
This completes the proof.
\end{proof}

\begin{lemma}
The spaces $C(S)$ are all contractible.
\end{lemma}
\begin{proof}
These spaces are naturally contained in $(\mathbb{R}^n)^A$, with respect to whose linear structure they are convex.
\end{proof}

\begin{proof}[Proof of Theorem~A]
The last three lemmas give us the first three homotopy equivalences in the computation
\[
\conf_A(\RR^n)
\simeq
\colim(C)
\simeq
\hocolim(C)
\simeq
\hocolim(\ast)
\simeq
B(\nOrd(A)),\]
where $\ast\colon\nOrd(A)\to\mathrm{Top}$ denotes the constant functor with value a point.  The last homotopy equivalence holds by definition.
\end{proof}

\section{The categories $\Theta_n$}
\label{Definitions:Section}

In this section we will recall Berger's inductive definition of the categories $\Theta_n$, and the description of the objects of $\Theta_n$ in terms of trees.
Except where noted, the material of this section is due to Berger~\cite{BergerTwo}.

\subsection{Segal's category of finite sets}
For $Z$ a finite set, denote its set of subsets by $\calP(Z)$.
Recall from \cite{Segal} that $\Gamma$ is the category whose objects are the finite sets, and in which a morphism
$\theta\colon X\to Y$
is a function
$\theta\colon X\to\calP(Y)$
with the property that $\theta({x_1})$ and $\theta({x_2})$ are disjoint when $x_1\neq x_2$.
Composition is defined by $(\phi\circ\theta) (s)=\bigcup_{t\in\theta(s)}\phi(t)$.
There is a natural functor
\[\gamma\colon\Delta\longrightarrow\Gamma.\]
It sends $[n]=\{0<\dots<n\}$ to $\mathbf{n}=\{1,\ldots,n\}$, and sends a morphism $f$ to the morphism $\gamma(f)$ defined by
\[i\longmapsto\{j \mid f(i-1)<j\leqslant f(i) \}.\]

\subsection{Wreath products}
The \emph{wreath product} $\Gamma\wr\calD$ of $\Gamma$ with an arbitrary category $\calD$ is defined as follows.
An object of $\Gamma\wr\calD$ is a symbol
\[X(D_x)\]
where $X$ is a finite set and $(D_x)_{x\in X}$ is a tuple of objects of $\calD$ indexed by $X$.
A morphism 
\[\theta\colon X(D_x)\longrightarrow Y(E_y)\]
consists of a morphism $\theta_\Gamma\colon X\to Y$ in $\Gamma$ and a morphism $\theta_{xy}\colon D_x\to E_y$ in $\calD$ whenever $y\in\theta_\Gamma(x)$.
Composition in $\Gamma\wr\calD$ is given by composition in $\Gamma$ and in $\calD$.  There is an apparent forgetful functor $\Gamma\wr \calD \to \Gamma$ given by $X(D_x) \mapsto X$.  

If $\calC$ is a category over $\Gamma$ then the \emph{wreath product} $\calC\wr\calD$ is defined to be the pullback $\calC\times_\Gamma (\Gamma\wr\calD)$.  This wreath product is functorial in the arguments $\calC \to \Gamma$ and $\calD$.  

The \emph{assembly} functor $\alpha\colon\Gamma\wr\Gamma\longrightarrow\Gamma$
is obtained by taking unions.
To be precise, an object $X(A_x)$ is sent to the disjoint union $\bigsqcup_{x\in X} A_x$ and 
a morphism $\theta\colon X(A_x)\to Y(B_y)$ is sent to the morphism
$\bigsqcup_{x\in X}A_x \to \bigsqcup_{y\in Y}B_y$ which assigns to
$a\in A_x$ the subset $\bigsqcup_{y\in\theta_\Gamma(x)}(\theta_{xy})(a)$.

\subsection{Wreath product with $\Delta$}\label{DeltaWreath:Subsection}
The Segal functor $\gamma\colon\Delta\to\Gamma$ allows us to define the wreath product $\Delta\wr\calD$ for any category $\calD$.
Unravelling the definitions above, we see that the objects of $\Delta\wr\calD$ are symbols of the form
\[[s](D_1,\ldots,D_s)\]
where $s\geqslant 0$ and $D_1,\ldots,D_s$ are objects of $\calD$.
A morphism in $\Delta\wr\calD$ 
\[f\colon [s](D_1,\ldots,D_s)\longrightarrow [t](E_1,\ldots,E_t)\]
consists of a morphism in $\Delta$
\[f_\Delta\colon [s]\longrightarrow[t]\]
and morphisms in $\calD$
\[f_{ij}\colon S_i\to T_j\]
for every pair $i,j$ satisfying $f(i-1)<j\leqslant f(i)$.

\subsection{The categories $\Theta_n$}

\begin{definition}\label{ThetaN:Definition}
The categories $\Theta_n$ are defined inductively by setting
\[\Theta_1=\Delta\qquad\mathrm{and}\qquad \Theta_{n}=\Delta\wr\Theta_{n-1}.\]
The \emph{assembly functors} $\gamma_n\colon\Theta_n\to\Gamma$ are defined inductively by setting $\gamma_1=\gamma$ and $\gamma_{n}=\alpha\circ(\gamma\wr\gamma_{n-1})$.
(The categories $\Theta_n$ were first introduced by Joyal \cite{JoyalDiscs}.
However, the above definition in terms of wreath products is due to Berger.)
\end{definition}

\subsection{The objects of $\Theta_n$}\label{Objects:Section}
\label{Objects:Subsection}
The objects of $\Theta_n$ are naturally identified with the planar level trees of height $n$.
(See Definition~\ref{PlanarLevelTrees:Definition}.)
When $n=1$ this identification sends the object $[s]$ to the tree with exactly $s$ non-root vertices, all with level $1$.
For $n>1$ the object $[s](S_1,\ldots,S_s)$ of $\Theta_n$ is identified with the planar level tree that has exactly $s$ vertices with level $1$, 
and in which the tree associated to $S_i$ appears as the subtree spanned by those vertices  for which there is a directed path to the $i$-th such level-$1$ vertex.

The value of the assembly functor $\gamma_n\colon\Theta_n\to\Gamma$ on a tree $T$ is naturally identified with the set of leaves of that tree.

\begin{example}
The object $[4]([2],[3],[0],[1])$ of $\Theta_2$ corresponds to the following planar level tree.
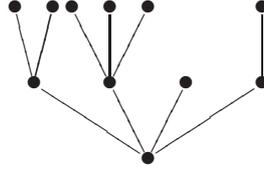
\[
\xy
(0,0)*{\bullet}="A";
(-15,10)*{\bullet}="B1";
(-5,10)*{\bullet}="B2";
(5,10)*{\bullet}="B3";
(15,10)*{\bullet}="B4";
(-17.5,20)*{\bullet}="C1";
(-12.5,20)*{\bullet}="C2";
(-10,20)*{\bullet}="C3";
(-5,20)*{\bullet}="C4";
(0,20)*{\bullet}="C5";
(15,20)*{\bullet}="C6";
"A";"B1" **\dir{-};
"A";"B2" **\dir{-};
"A";"B3" **\dir{-};
"A";"B4" **\dir{-};
"B1";"C1" **\dir{-};
"B1";"C2" **\dir{-};
"B1";"C2" **\dir{-};
"B2";"C3" **\dir{-};
"B2";"C4" **\dir{-};
"B2";"C5" **\dir{-};
"B4";"C6" **\dir{-};
\endxy
\] 
\end{example}

\section{Morphisms in $\Theta_n$}
\label{Morphisms:Section}

From the category-theoretic point of view this section is the technical heart of the paper.

We have just seen that the objects of $\Theta_n$ admit a simple description as the planar level trees of height $n$.
The morphisms in $\Theta_n$ are much less easy to describe from this point of view.
This is indicated by the assembly functor $\gamma_n\colon\Theta_n\to\Gamma$, which sends a tree to its set of leaves with level $n$.
Just as a morphism in $\Gamma$ is not a map of sets, so a morphism in $\Theta_n$ is not a map of trees in any obvious sense.

In this section we will define \emph{healthy} objects and \emph{active} morphisms in $\Theta_n$, and we will give a simple, combinatorial description of the set of active morphisms $S\to T$ when $T$ is healthy.
In the next section this description will be applied to the category $\Theta_n(A)$.

\subsection{Active morphisms and healthy trees}

\begin{definition}
An object of $\Theta_n$ is \emph{healthy} if the corresponding planar level tree is healthy, or in other words, has no leaves at level $1,\ldots,n-1$.
\end{definition}

\begin{definition}
A morphism $\theta\colon X\to Y$ in $\Gamma$ is \emph{active} if $\bigcup_{x\in X}\theta(x)=Y$.
A morphism in $\Theta_n$ is \emph{active} if its image under $\gamma_n$ is active.
(Our definition of active morphisms in $\Gamma$ corresponds to Lurie's notion of active morphism in the category of based finite sets, which is the opposite of $\Gamma$ \cite{LurieHA}.)
\end{definition}

By way of section~\textsection\ref{Objects:Section}, from now on we will not distinguish between an object of $\Theta_n$ and a planar level tree of height $n$.

\subsection{The branching condition}

Now we introduce some notation regarding objects and morphisms in $\Theta_n$.
These are closely related to the definitions introduced in section~\ref{nOrd:Section}.

\begin{definition}\label{OrderBranchingLevels:DefinitionTwo}
Let $T$ be an object of $\Theta_n$.
Then the planar structure of $T$ endows the set of level-$n$ leaves $\gamma_n(T)$ with a linear ordering that we denote $<_T$.
Given $a,b\in\gamma_n(T)$ we define the \emph{branching level}
\[ b_T(a,b)\]
to be the level of the vertex at which the directed paths from $a$ and $b$ to the root meet.
Note that the branching levels between consecutive elements of $\gamma_n(T)$ determine all the other branching levels.
\end{definition}

\begin{definition}\label{BranchingCondition:DefinitionTwo}
Let $S$, $T$ be objects of $\Theta_n$, with $T$ healthy.
An active morphism $\bar f\colon \gamma_n(S)\to \gamma_n(T)$ satisfies the \emph{branching condition} if, for all quadruples $a,b,c,d$ with $a,b\in\gamma_n(S)$ and $c\in\bar f(a)$, $d\in\bar f(b)$, the following condition holds:
\begin{quotation}
$b_T(c,d)\leqslant b_S(a,b)$, with equality only if the order of $c,d$ in $T$ is the same as the order of $a,b$ in $S$.
\end{quotation}
The condition on orderings means that ($c<_T d$ and $a<_S b$) or ($d<_T c$ and $b<_S a$).
\end{definition}

This branching condition is closely related to the one that describes when there is a morphism in $\nOrd(A)$.
Here, however, there is no need for the map $\bar f$ to be an isomorphism.
So for fixed $a,b\in\gamma_n(S)$ there may be several choices of elements $c\in\bar f(a)$, $d\in\bar f(b)$.

\subsection{Active morphisms into healthy trees}

The next theorem fully characterises active morphisms into healthy objects in terms of the branching condition. 

\begin{theorem}\label{Morphisms:Theorem}
Let $S$ and $T$ be trees in $\Theta_n$ with $T$ healthy.
Then the assignment
\[f\longmapsto \gamma_n(f)\]
determines a bijection between active morphisms $f\colon S\to T$ and active morphisms $\bar f\colon\gamma_n(S)\to\gamma_n(T)$ satisfying the branching condition.
\end{theorem}

The theorem will be applied in the next section to give us a description of the category $\Theta_n(A)$.
For that application we only need the case when $\gamma_n(f)$ is an isomorphism.
However, the theorem proceeds by induction on $n$, and at the induction step we are forced to pass from isomorphisms to general active morphisms.

\subsection{Proof of Theorem~\ref{Morphisms:Theorem}}
Our proof of the theorem consists of the next three lemmas, all of which are proved by induction on $n$.
We use the notation of section~\ref{DeltaWreath:Subsection} throughout.

\begin{lemma}
Let $f,g\colon S\to T$ be active morphisms in $\Theta_n$ with $T$ healthy and $\gamma_n(f)=\gamma_n(g)$.
Then $f=g$.
\end{lemma}
\begin{proof}
In the case $n=1$ it is trivial to check that an active morphism in $\Delta$ is determined by its image in $\Gamma$.

In the general case, if $T$ is the trivial tree then there is nothing to prove.
So we assume that all leaves of $T$ are at level $n$.
Writing
\[
\gamma_n(S)=\gamma_{n-1}(S_1)\cup\cdots\cup\gamma_{n-1}(S_s)
\qquad\mathrm{and}\qquad
\gamma_n(T)=\gamma_{n-1}(T_1)\cup\cdots\cup\gamma_{n-1}(T_t),
\]
we find that if $j\in \gamma(f_\Delta)(i)$, then $\gamma_{n-1}(T_j)$ must lie in $\gamma_{n}(f)(\gamma_{n-1}(S_i))$.
Since each $\gamma_{n-1}(T_j)$ is nonempty, this means that $\gamma_n(f)$ determines $\gamma(f_\Delta)$.
The same reasoning holds for $g$, so that we have $\gamma(f_\Delta)=\gamma(g_\Delta)$ and so $f_\Delta=g_\Delta$.

Now $f$ and $g$ are determined by \emph{active} morphisms $f_{ij},g_{ij}\colon S_i\to T_j$ in $\Theta_{n-1}$ for which each $T_j$ is healthy
and that satisfy $\gamma_{n-1}(f_{ij})=\gamma_{n-1}(g_{ij})$.
By the induction hypothesis we have $f_{ij}=g_{ij}$, and the proof is complete.
\end{proof}

\begin{lemma}
Let $f\colon S\to T$ be an active morphism to a healthy tree.
Then $\gamma_n(f)$ satisfies the branching condition.
\end{lemma}
\begin{proof}
Let $a,b,c,d$ be as in the statement of the branching condition.
Suppose that $a\in\gamma_{n-1}(S_i)$ and $b\in\gamma_{n-1}(S_{j})$ and $c\in\gamma_{n-1}(T_{i'})$ and $d\in\gamma_{n-1}(T_{j'})$.
Without loss $i\leqslant j$.
There are now three possibilities, in each of which we will verify that $a,b,c,d$ satisfy the necessary condition.

First, $i< j$, so that $b_S(a,b)=0$ and $a<_Sb$.
Since $i'\in \gamma(f_\Delta)(i)$ and $j'\in \gamma(f_\Delta)(j)$, we must have $i'<j'$, so that $b_T(c,d)=0$ and $c<_Td$.
So the condition holds in this case.

Second, $i=j$ but $i'\neq j'$.
Then $b_S(a,b)\geqslant 1$ while $b_T(c,d)=0$, so the branching condition holds.

Third, $i=j$ and $i'=j'$.
Then $b_S(a,b)=b_{S_i}(a,b)+1$ and $b_T(c,d)=b_{T_{i'}}(c,d)+1$, and $f_{ii'}\colon S_i\to T_{i'}$ is an active morphism for which $c\in\gamma_{n-1}(f_{ii'})(a)$ and $d\in\gamma_{n-1}(f_{ii'})(b)$.
So it suffices to check that the required condition holds when $f$ is replaced by $f_{ii'}$, but this follows from the induction hypothesis.
\end{proof}

\begin{lemma}
Let $S$ and $T$ be objects of $\Theta_n$ with $T$ healthy, and let $\bar f\colon \gamma_n(S)\to \gamma_n(T)$ be an active morphism in $\Gamma$ that satisfies the branching condition.
Then there is $f\colon S\to T$ in $\Theta_n$ such that $\gamma_n(f)=\bar f$.
\end{lemma}

\begin{proof}
Again, we proceed by induction on $n$, the case $n=1$ being a trivial observation about Segal's functor $\gamma\colon\Delta\to\Gamma$.

If $t=0$ then there is nothing to prove, so we assume $t>0$.
Writing 
\[\gamma_n(S)=\gamma_{n-1}(S_1)\cup\cdots\cup\gamma_{n-1}(S_s)
\quad \mathrm{and} \quad
\gamma_n(T)=\gamma_{n-1}(T_1)\cup\cdots\cup\gamma_{n-1}(T_t),\]
we see first that the branching condition means first that each $\bar f (\gamma_{n-1}(S_i))$ is the union of certain of the $\gamma_{n-1}(T_j)$, and second that this union in fact has the form
$\gamma_{n-1}(T_{j})\cup\cdots\cup\gamma_{n-1}(T_{j'})$
for some $j\leqslant j'$.
We may therefore find $0= r_0\leqslant \cdots \leqslant r_{s}=t$ such that 
$\bar f (\gamma_{n-1}(S_i))=\gamma_{n-1}(T_{r_{i-1}+1})\cup\cdots\cup\gamma_{n-1}(T_{r_{i}})$.

Now we can write down a morphism $f_\Delta\colon[s]\to[t]$ and,
by the inductive hypothesis,
morphisms $\bar f_{ij}\colon\gamma_{n-1}(S_i)\to\gamma_{n-1}(T_j)$ whenever $j\in \gamma(f_\Delta)(i)$, with the property that for $x\in\gamma_{n-1}(S_i)$ we have
\[\bar f(x)=\bigcup_{j\in\gamma(f_\Delta)(i)}\bar f_{ij}(x).\]
It is simple to check that each $\bar f_{ij}$ is active and satisfies the branching condition for a morphism $S_i\to T_j$, and thus by the induction hypothesis there is a morphism $f_{ij}\colon S_i\to T_j$ with $\gamma_{n-1}f_{ij}=\bar f_{ij}$.
Now $f_\Delta$ and the $f_{ij}$ define the required morphism $f$.
\end{proof}

\section{The category $\Theta_n(A)$}
\label{ThetanA:Section}

We now use the results of the last section to study the categories  $\Theta_n(A)$, and prove Theorem~B.

\subsection{Morphisms in $\Theta_n(A)$}
Recall from the introduction that for a finite set $A$, the category $\Theta_n(A)$ consists of pairs $(S,\sigma)$ where $S$ is an object of $\Theta_n$ and $\sigma\colon\gamma_n(S)\to A$ is an isomorphism.
A morphism $f\colon (S,\sigma)\to (T,\tau)$ in $\Theta_n(A)$ is a morphism $f\colon S\to T$ in $\Theta_n$ for which $\tau\circ\gamma_n(f)=\sigma$.
From this point we will suppress the isomorphism $\sigma$ from the notation.

In light of Section~\ref{Objects:Subsection}, the objects of $\Theta_n(A)$ may be described as planar level trees of height $n$, whose leaves at level $n$ are labelled in bijection with $A$.
So an object $S$ of $\Theta_n(A)$ determines an {ordering} $<_S$ on $A$ and branching levels $b_S(a,b)$ for all $a,b\in A$, exactly as in Definition~\ref{OrderBranchingLevels:DefinitionOne}.
Moreover, it makes sense to ask whether the branching condition holds for a morphism $S\to T$ in $\Theta_n(A)$, exactly as in Definition~\ref{BranchingCondition:DefinitionOne}.

Comparing with Definitions~\ref{OrderBranchingLevels:DefinitionTwo} and~\ref{BranchingCondition:DefinitionTwo} and Theorem~\ref{Morphisms:Theorem}, we immediately obtain the following.

\begin{corollary}\label{ThetaNAMorphisms:Corollary}
Let $S$ and $T$ be objects of $\Theta_n(A)$ with $T$ healthy.
Then there is at most one morphism $S\to T$, and it exists if and only if the branching condition holds.
\end{corollary}

\subsection{Proof of Theorem~B}

Recall that Theorem~B states that there is a full embedding $\nOrd(A)\hookrightarrow\Theta_n(A)$ that induces a homotopy equivalence on geometric realisations.

The objects of $\nOrd(A)$ are the healthy planar level trees of height $n$ equipped with a labelling of their leaves in bijection with $A$.
The objects of $\Theta_n(A)$ have exactly the same description, except that the tree need not be healthy.
This gives us an inclusion $i\colon \nOrd(A)\hookrightarrow\Theta_n(A)$ on objects, and by comparing Corollary~\ref{ThetaNAMorphisms:Corollary} with Definition~\ref{BranchingCondition:DefinitionOne} we see that it is a full functor.

Let $S$ be an object of $\Theta_n(A)$.
Denote by $S^h$ the healthy subtree spanned by the level-$n$ vertices.  
Then it is easily seen that the branching condition for a morphism $S\to T$ is identical to the branching condition for a morphism $S^h\to T$.
It follows that there is a morphism $S\to S^h$ in $\Theta_n(A)$, with target in $\nOrd(A)$, and which is initial among all morphisms from $S$ to an object of $\nOrd(A)$.

Now the assignment $S\mapsto S^h$ determines a functor $r\colon
\Theta_n(A)\to\nOrd(A)$ for which $r\circ i=1$ and for which there is a natural transformation $1\Rightarrow i\circ r$.
It follows that the maps $B(\nOrd(A))\to B(\Theta_n(A))$ and $B(\Theta_n(A))\to B(\nOrd(A))$ induced by $i$ and $r$ respectively are homotopy-inverse to one another.
This completes the proof of Theorem~B.

\bibliographystyle{plain}
\bibliography{ThetanBibliography}
\end{document}